\documentclass{article}

\usepackage{amsmath}
\usepackage{amsfonts}
\usepackage{amssymb}
\usepackage{amsthm}
\usepackage{graphicx}
\usepackage{color}
\usepackage{hyperref}
\usepackage{enumitem}
\usepackage{bm}
\usepackage{mathtools}
\usepackage{enumitem}
\usepackage{leftidx}
\usepackage{isomath}
\usepackage{cite}

\newcommand{\remove}[1] {}
\newtheorem{definition}{Definition}
\newtheorem{lemma}{Lemma}
\newtheorem{corollary}{Corollary}
\newtheorem{theorem}{Theorem}

\newtheorem{claim}{Claim}
\newtheorem{remark}{Remark}
\newtheorem*{summary*}{Summary of results}
\newcommand\scaledbar[1]{\bar{#1}}
\newtheorem{example}{Example}

\newtheorem{innercustomthm}{Theorem}
\newenvironment{customthm}[1]
  {\renewcommand\theinnercustomthm{#1}\innercustomthm}
  {\endinnercustomthm}
  
  \newtheorem*{dich-thm*}{Dichotomy Theorem}

\newcommand{\n}{\{1,\ldots,n\}}

\newcommand{\f}{\bar{f}}

\newcommand{\un}{uniformly}
\newcommand{\upd}{upd}

\newcommand{\eat}[1]{}

\date{August 31, 2016}

\author {
{ Lefteris Kirousis}\thanks{The research of the first author was carried in part while he was visiting the Computer Science Department of UC Santa Cruz and   was partially  by the European Union (European Social Fund ESF) and Greek national funds through the Operational Program ``Education and Lifelong Learning'' of the National Strategic Reference Framework (NSRF) - Research
Funding Program: ARISTEIA II.}\\[0.1cm] Department of Mathematics\\National and Kapodistrian University of Athens\\ and \\ Computer Technology Institute and Press ``Diophantus"\\
 \href{mailto:lkirousis@math.uoa.gr}{lkirousis@math.uoa.gr}\vspace{\baselineskip}
\and { Phokion G. Kolaitis}
\\[0.1cm]Computer Science Department\\UC Santa Cruz and IBM Research - Almaden \\ \href{mailto:kolaitis@cs.ucsc.edu}{kolaitis@cs.ucsc.edu} \vspace{\baselineskip}
\and {John Livieratos}\\[0.1cm] Department of Mathematics\\National and Kapodistrian University of Athens\\ 
 \href{mailto:jlivier89@math.uoa.gr}{jlivier89@math.uoa.gr}}

\title{Aggregation of Votes \\ with Multiple Positions on Each Issue}

\begin{document}

\maketitle
\begin{abstract}
We consider the problem of aggregating votes cast by a society on a fixed set of issues, where each member of the society may vote for one of several positions on each issue, but  the combination of votes on the various issues is restricted to a  set of feasible voting patterns. We require  the aggregation to be supportive, i.e. for every issue $j$ the corresponding component $f_j$ of every  aggregator on every issue   should satisfy $f_j(x_1, ,\ldots, x_n) \in \{x_1, ,\ldots, x_n\}$. We prove  that, in such a set-up, non-dictatorial aggregation of votes in  a society of some size is  possible if and only if either non-dictatorial aggregation is possible  in a society of only two members  or  
a {\em ternary} aggregator exists  that either on every issue $j$ is a majority operation, i.e. the corresponding component satisfies $f_j(x,x,y) = f_j(x,y,x) = f_j(y,x,x) =x, \forall x,y$, or on every issue is a minority operation, i.e. the corresponding component satisfies $f_j(x,x,y) = f_j(x,y,x) = f_j(y,x,x) =y, \forall x,y.$ We then introduce a notion of {\em \un \/} non-dictatorial aggregator, which is defined to be an aggregator   that on every issue, and when restricted to an arbitrary  two-element subset of the votes for that issue, differs from all projection functions. We first give a  characterization of sets of feasible voting patterns that admit a {\un} non-dictatorial aggregator. Then   making use of Bulatov's dichotomy theorem for conservative constraint satisfaction problems,  we connect social choice theory with combinatorial complexity by proving that if a  set of feasible voting patterns $X$ has a {\un}  non-dictatorial aggregator of some arity  then  the multi-sorted conservative constraint satisfaction problem on $X$, in the sense introduced by  Bulatov and Jeavons, with each issue representing a sort,  is tractable; otherwise it is NP-complete.  
 	
\end{abstract}

\section{Introduction} \label{intro:sec}
Kenneth Arrow initiated the theory of aggregation by establishing his celebrated General Possibility Theorem (also known as Arrow's Impossibility Theorem) \cite{arrow1951social}, which asserts  that it is impossible, even in easy cases, to aggregate in a non-trivial (non-dictatorial) way the preferences of a society. Wilson \cite{wilson1975theory} introduced aggregation  on  general attributes, rather than just preferences,  and proved Arrow's result in this context.  Later on, Dokow and Holzman \cite{dokow2010aggregation} adopted  a framework similar to Wilson's in which
 the voters have a binary position  on a number of issues, and an individual voter's  feasible position patterns  are restricted to lie in a domain $X$. Dokow and Holzman discovered a necessary and sufficient condition for  $X$ to have  a non-dictatorial aggregation that  involves  a property called  {\em total blockedness}, which was originally introduced in \cite{nehring2002stategy}.  Roughly speaking, a domain $X$ is totally blocked if ``any position on any issue can be deduced from any position on any issue" (the precise  definition is given in  Section \ref{totbl} of the present paper). In other words,  total blockedness is a property that refers to the propagation of individuals'  positions from one issue to another.

After this, Dokow and Holzman \cite{dokow2010aggregationnonB} extended their earlier work by allowing the positions to be  non-Boolean (non-binary).
  By extending total blockedness to the non-Boolean framework, they gave a sufficient  condition for the possibility of non-dictatorial aggregation; moreover, they gave a weaker necessary condition, which however may not be  sufficient.

Recently, Szegedy and Xu \cite{szegedy2015impossibility} discovered  necessary and sufficient conditions for non-dictatorial aggregation. Quite remarkably, their approach  relates aggregation theory with
 universal algebra, specifically with the structure of the space of \emph{polymorphisms}, that is, functions under which a relation is closed. It should be noted that properties of polymorphisms have been successfully
  used towards the delineation of the boundary between tractability and intractability for the Constraint Satisfaction Problem  (for an overview, see, e.g., \cite{bulatov2008recent}).

Szegedy and Xu  \cite{szegedy2015impossibility} distinguished the \emph{supportive} (\emph{conservative}, in the terminology of complexity theory) case (studied by Dokow and Holzman  in \cite{dokow2010aggregationnonB}), where the social  position must be equal to the position of at least one individual,  from the \emph{idempotent} (\emph{Paretian}) case, where the social position  may  not agree with any individual position, unless the individuals  are unanimous. In the idempotent case, they gave a necessary and sufficient condition for possibility of non-dictatorial aggregation that involves no   propagation criterion like total blockedness, but only refers to the possibility of non-trivial aggregation for societies of a fixed cardinality (as large as the space of positions). In the  supportive case, however, their necessary and sufficient conditions still involve the notion of a total blockedness.

Here,
we follow Szegedy and Xu's  idea of deploying the algebraic ``toolkit" \cite{szegedy2015impossibility} and 
we prove  that, in the supportive case, non-dictatorial aggregation is possible for  all societies of  some cardinality 
if and only if non-dictatorial aggregation is possible in a society of just two members, or 
a {\em ternary} aggregator exists  that either on every issue $j$ is a majority operation, i.e. the corresponding component $f_j$ satisfies $$f_j(x,x,y) = f_j(x,y,x) = f_j(y,x,x) =x, \forall x,y,$$ or on every issue is a minority operation, i.e. the corresponding component satisfies $$f_j(x,x,y) = f_j(x,y,x) = f_j(y,x,x) =y, \forall x,y.$$ 
For the notions of majority and minority operations see Szendrei \cite[p. 24]{szendrei1986clones}.

 Also, in Section \ref{totbl} we characterize total blockedness not as a property  of propagation of individual positions but simply as a weak form of impossibility domain (no binary non-dictatorial aggregator).   

In the sequel (Section \ref{upd:sect}), we  introduce the notion of {\em \un \/} non-dictatorial aggregator, which is defined to be an aggregator   that on every issue, and when restricted to an arbitrary  two-element subset of the votes for that issue, differs from all projection functions. We first give a characterization of sets of feasible voting patterns that admit {\un} non-dictatorial aggregators. Then making  use  of Bulatov's dichotomy theorem for conservative  Constraint Satisfaction Problems (CSP's), see \cite{bulatov2011complexity, barto2011dichotomy, bulatov2016conservative}, we connect social choice theory with combinatorial complexity by proving that if a  set of feasible voting patterns $X$ has a {\un}  non-dictatorial aggregator of some arity  then  the multi-sorted conservative constraint satisfaction problem on $X$, in the sense introduced by  Bulatov and Jeavons \cite{bulatov2003algebraic}, with each issue representing a sort,  is tractable; otherwise it is NP-complete.

\section{Preliminaries and Summary of Results}\label{prelims}

\subsection{Basic Concepts}
In all that follows, we have a fixed set  $I = \{1, ,\ldots,m\}$ of issues.
Let $\mathcal{A}= \{A_1, \ldots, A_m\}$ be a family of  finite sets, each of cardinality at least 2,  representing the possible positions (voting options) on the issues $1, \ldots, m$, respectively.     When every $A_j$  has cardinality exactly 2, i.e. when for every issue only a ``yes" or ``no" vote is allowed, we say that we are in the {\em binary} (or the  {\em Boolean}) {\em framework}.

Let $X$ be a non-empty subset of $\prod_{j=1}^m A_j$ that represents the feasible  voting patterns.
We write
$X_j$, $j =1 \ldots, m$,  to denote the $j$-th projection of $X$.
From now on, we assume that each $X_j$ has cardinality at least 2 (this is a \emph{non-degeneracy} condition). Throughout  the rest of the paper, unless otherwise declared,  $X$ will denote a set  of feasible voting patterns on $m$ issues, as we just described.

Let $n \geq 2$ be an integer representing the number of voters. The elements of
 $X^n$ can be viewed as $n\times m$  matrices, whose rows correspond to voters and whose columns correspond to issues.  We  write $x^i_j$ to denote the entry of the matrix in row $i$ and column $j$; clearly, it stands for the vote of voter $i$ on issue $j$.   The row vectors of such matrices will be denoted as  $x^1, \ldots, x^n$, and the column vectors as $x_1, \ldots, x_m$.  

\paragraph{Aggregators and Possibility Domains}
Let $\bar{f} = (f_1, \ldots, f_m)$ be an $m$-tuple of $n$-ary functions $f_j: A_j^n \mapsto A_j$.

 An $m$-tuple of functions $\bar{f} = (f_1, \ldots, f_m)$ as above is called  {\em supportive (conservative)} if
 for all $j = 1 \ldots  m$, we have that \[\text{ if } x_j = (x_j^1, \ldots, x_j^n) \in A_j^n, \text{ then } f_j(x_j) = f_j(x_j^1, \ldots, x_j^n) \in \{x_j^1, \ldots, x_j^n\}.\]

An $m$-tuple $\bar{f} = (f_1, \ldots, f_m)$ of ($n$-ary)  functions as above is called an ($n$-ary) {\em aggregator for $X$} if it is supportive and, for all $j = 1,\ldots,  m$ and for all $x_j \in A_j^n, j=1, \ldots, m$, we have that 
\[\text{ if } (x^1, \ldots, x^n ) \in X^n, \text{ then }(f_1(x_1), \ldots, f_m(x_m)) \in X.\]
Note that  $(x^1, \ldots, x^n )$ is an $n\times m$ matrix with rows $x^1, \ldots, x^n $ and columns $x_1, \ldots, x_m, $ whereas $(f_1(x_1), \ldots, f_m(x_m))$ is a row vector required to be in $X$.

An aggregator  $\bar{f} = (f_1, \ldots, f_m)$ is called  {\em trivial (dictatorial) on $X$} if there is a number $d \in \{1, \ldots, n\}$
such that
$(f_1,\ldots f_m)\restriction {X} = ({\rm pr}_d^n,\ldots,{\rm pr}_d^n)\restriction{X}$, i.e., $(f_1,\ldots f_m) $ restricted to $X$ is equal to $({\rm pr}_d^n,\ldots,{\rm pr}_d^n)$ restricted to $X$,
  where ${\rm pr}_d^n$ is  the $n$-ary projection on the $d$-th coordinate; otherwise, $\bar{f}$ is called  {\em non-trivial (non-dictatorial) on $X$}. We say that $X$ {\em has a non-trivial aggregator} if, for some $n\geq 2$, there is  a non-trivial $n$-ary aggregator on $X$.

A set $X$ of feasible voting patterns is called a {\em possibility domain} if it has a non-trivial aggregator. Otherwise, it is called an {\em impossibility domain}. A possibility domain is, by definition,  one where aggregation is possible for societies of some cardinality, namely, the arity of the non-dictatorial aggregator.

Aggregators do what their name indicates, that is, they aggregate positions  on $m$ issues, $j=1, \ldots, m$, from data representing the voting patterns of $n$ individuals on all issues.
  The fact that aggregators are defined as $m$-tuples of functions $A_j^n \mapsto A_j$,  rather than a single function $X^n \mapsto X$, reflects the fact that the social vote is assumed to be extracted issue-by-issue, i.e., the aggregate vote on each issue does not depend on voting data on other issues. The fact that aggregators are assumed to be supportive (conservative) reflects the restriction of our model that the social vote for every issue should be equal to  the vote cast on this issue by at least one individual. Finally, the requirement of non-triviality for aggregators reflects the fact that the aggregate vote  should not be extracted by adopting the vote of a single individual  designated as a ``dictator".

\begin{example}
Suppose that $X$ is a cartesian product $X= Y \times Z$, where $Y\subseteq \prod_{j=1}^l A_j$ and $Z \subseteq \prod_{j=l+1}^{m} A_j$, with $1 \leq l <m$. It is easy to see that $X$ is a possibility domain.\end{example}
 Indeed,  for every $n\geq 2$,  $X$ has non-trivial $n$-ary aggregators of the form
$(f_1,\ldots, f_l,f_{l+1},\ldots,f_m)$, where for some   $d$ and $d'$ with $d\not = d'$, we have that for $j=1, \ldots, l,  f_j= {\rm pr}^n_d$ and for $j=l+1, \ldots, m, f_j = {\rm pr}^n_{d'} $.
Thus, every cartesian product of feasible patterns is a possibility domain. \hfill $\Box$

\remove{
\paragraph{Polymorphisms}
Let $X \subseteq A^m$ be an $m$-ary relation on a set $A$; in this case, $X$ can be thought of as a set of feasible voting patterns on $m$ issues, where the voting options are the same for all issues. Let $f: A^n\rightarrow A$ be a $n$-ary function on $A$. We say that $X$ is {\em closed under} $f$ if for every $n$-tuple $x^1= (x_1^1,\ldots,x^1_m), \ldots, x^n=(x^n_1,\ldots, x^n_m)$ of elements of $X$, the coordinate-wise application $(f(x_1^1,\ldots,x^n_1),\ldots, f(x^1_m,\ldots,x^n_m)$
of $f$ to this $n$-tuple is an element of $X$. We also say that $f$ is a {\em polymorphism} of $X$.
 If $f$ is supportive (i.e., $f(w_1,\ldots,w_n) \in \{w_1,\ldots,w_n\}$, for every $(w_1,\ldots,w_n) \in A^n)$ and $X$ is closed under $f$, then the $m$-tuple
$\bar{f} = (f,\ldots, f)$ is an aggregator for $X$. Clearly, every set is closed under every projection function, which gives rise to trivial aggregators.  If, however, a set $X$ is closed under a supportive function $f$ that is not a projection, then $\bar{f} = (f,\ldots, f)$ is a non-trivial aggregator for $X$, hence $X$ is a possibility domain.  Next, we will consider two types of such possibility. }

Now, 
following \'A. Szendrei \cite[p. 24]{szendrei1986clones}, we define:\begin{definition}
A ternary operation $f: A^3 \mapsto A $ on an arbitrary set $A$ is  a {\em majority} operation if 
\[f(x,x,y) = f(x,y,x) = f(y,x,x) =x, \forall x,y \in A,\]  and  it is a {\em minority} operation  if  \[f(x,x,y) = f(x,y,x) = f(y,x,x) =y, \forall x,y \in A.\]
\end{definition}
We also define
\begin{definition}
A set of feasible voting patterns $X$ as before admits a  {\em majority aggregator} if it admits a ternary aggregator 	$\bar{f} = (f_1,\ldots, f_m)$ such that for all $j=1, \ldots, m$, $f_j$ is a majority operation on $X_j$. $X$ admits a {\em minority aggregator} if it admits a ternary aggregator 	$\bar{f} = (f_1,\ldots, f_m)$ such that for all $j=1, \ldots, m$, $f_j$ is a minority operation on $X_j$.
\end{definition}

Clearly $X$ admits a majority aggregator iff  there is a ternary aggregator $\bar{f} = (f_1, \ldots, f_m)$ for $X$ such that, for  all $j=1, \ldots, m$
and for  all two-element subsets $B_j \subseteq X_j$, we have that
$$f_j{\restriction{B_j}} = {\rm maj},$$ where

\begin{equation*}
	{\rm maj} (x, y, z) = \begin{cases}
 x & \text{ if }  x= y \text{ or } x=z,\\
 y & \text{ if } y = z.
 \end{cases}
\end{equation*}
Also $X$ admits a minority aggregator if there is a ternary $\bar{f} = (f_1, \ldots, f_m)$ for $X$ such that, for  all $j=1, \ldots, m$
and for  all two-element subsets $B_j \subseteq X_j$, we have that
$$f_j{\restriction{B_j}} = \oplus,$$ where

\begin{equation*}
	\oplus (x, y, z) = \begin{cases}
 z & \text{ if }  x= y,\\
 x & \text{ if } y = z,\\
 y & \text{ if } x = z. 
 \end{cases}
\end{equation*}

Clearly as well, in the Boolean framework (for all issues only ``yes" or ``no" votes are allowed) $X$  admits a majority aggregator iff $X$ as a logical relation (as a subset of $\{0,1\}^m$) is  bijunctive  and $X$ admits a minority aggregator iff $X$ as a logical relation is affine (see Schaefer \cite{schaefer1978complexity}).

\begin{example} The set
$$X=\{(a,a,a), (b,b,b), (c,c,c), (a,b,b), (b,a,a), (a,a,c), (c,c,a)\}$$
admits a majority aggregator.
\end{example}
To see this,
let $\bar{f}=(f,f,f)$, where $f:\{a,b,c\} \rightarrow \{a,b,c\}$ is as follows:
\begin{displaymath}
f(u,v,w)= \left\{ \begin{array}{l l} \vspace{1mm}
a  & \textrm{if  $u$, $v$, and $w$ are pairwise different}; \\ \vspace{1mm}
{\rm maj}(u,v,w) & \textrm{otherwise}.
\end{array} \right.
\end{displaymath}
Clearly, if $B$ is a two-element subset of $\{a, b, c\}$, then $f\restriction{B} = {\rm maj}$. So, to show that $X$ admits a majority aggregator, it remains to show that $\bar{f}=(f,f,f)$ is an aggregator for $X$. In turn, this amounts to showing that $\bar{f}$ is supportive and that $X$ is closed under $f$. It is easy to check that $\bar{f}$ is
supportive. To show that $X$ is closed under $f$, let
let $x=(x_1,x_2,x_3), y=(y_1,y_2,y_3), z=(z_1,z_2,z_3)$ be three elements of $X$.
 We have to show that $(f(x_1,y_1,z_1), f(x_2,y_2,z_2), f(x_3,y_3,z_3))$ is also in $X$. The only case that needs to be considered is when $x$, $y$, and $z$ are distinct. There are several subcases that need to be considered. For instance, if
  $x= (a,b,b)$, $y = (a, a, c)$, $z= (c,c,a)$, then $(f(a,b,b), f(a,a,c), f(c,c,a))= (a,a,a) \in X$; the remaining combinations are left to the reader.
 \hfill  $\Box$
\begin{example} The set
$X=\{(a,b,c),(b,a,a),(c,a,a)\}$  admits a minority aggregator.
\end{example}

To see this,
let $\bar{f}=(f,f,f)$, where $f:\{a,b,c\} \rightarrow \{a,b,c\}$ is as follows:
\begin{displaymath}
f(u,v,w)= \left\{ \begin{array}{l l} \vspace{1mm}
a  & \textrm{if  $u$, $v$, and $w$ are pairwise different}; \\ \vspace{1mm}
\oplus(u,v,w) & \textrm{otherwise}.
\end{array} \right.
\end{displaymath}
Clearly, if $B$ is a two-element subset of $\{a, b, c\}$, then $f\restriction{B} = \oplus$. So, to show that $X$ admits a minority aggregator, it remains to show that $\bar{f}=(f,f,f)$ is an aggregator for $X$. In turn, this amounts to showing that $\bar{f}$ is supportive and that $X$ is closed under $f$. It is easy to check that $\bar{f}$ is
supportive. To show that $X$ is closed under $f$, let
let $x=(x_1,x_2,x_3), y=(y_1,y_2,y_3), z=(z_1,z_2,z_3)$ be three elements of $X$.
 We have to show that $(f(x_1,y_1,z_1), f(x_2,y_2,z_2), f(x_3,y_3,z_3))$ is also in $X$. The only case that needs to be considered is when $x$, $y$, and $z$ are distinct, say, $x= (a,b,c)$, $y = (b, a, a)$, $z= (c,a,a)$. In this case, we have that $(f(a,b,c), f(b,a,a), f(c,a,a))= (a,b,c) \in X$; the remaining combinations are similar.   \hfill $\Box$

\smallskip
\remove{
It is interesting to compare the concepts of admitting a majority or a minority aggregator  with the criterion of tractability of multi-sort conservative constraint satisfaction problems given by Bulatov in \cite[Theorem 2.16]{bulatov2011complexity}. Actually, that criterion refers also to the operations   
$\oplus$ and ${\rm maj}$, but, more importantly, the aggregator $\bar{f}$ in Bulatov's work depends on $B_j$, whereas above we assume that there is a single $\bar{f}$ that works for all $B_j$'s.
}

So far, we have given examples of possibility domains only. We close this section with an example of an impossibility domain in the Boolean framework. 
\begin{example}\label{ex:1-in-3}
Let set $W=\{(1,0,0), (0,1,0), (0,0,1)\}$ be the set  of all Boolean tuples of length $3$ in which exactly one $1$ occurs. $W$ is an impossibility domain. 
\end{example}
Indeed, it can be seen that $W$ is not affine and does not admit a binary non-dictatorial aggregator. Therefore by  Theorem \ref{Bbasicthm} in the next subsection, $W$ is an impossibility domain. \hfill$\Box$

Note that, in the context of generalized satisfiability problems studied by Schaefer \cite{schaefer1978complexity}, the set $W$ gives rise to the NP-complete problem {\sc Positive 1-in-3-Sat}. On the other hand, as discussed earlier, the cartesian product $W\times W$ is a possibility domain. Using the results in \cite{schaefer1978complexity}, however,  it can be verified  that the generalized satisfiability problem arising from $W\times W$ is NP-complete. On the other hand, e.g., $\{0,1\}^m$ is trivially a possibility domain and gives rise to a tractable satisfiability problem.  Thus, the property of $X$ being a  possibility domain is not related to the tractability  of the generalized satisfiability problem arising from $X$.  Nevertheless in Section \ref{upd:sect} we  exhibit  the equivalence  between a stronger, {\em uniform}, notion of $X$ being a possibility domain and the weaker notion of the tractability of the {\em multi-sorted} generalized satisfiability problem arising from $X$, where each issue is taken as a sort.  Actually, we prove this equivalence not only for satisfiability but for CSP's whose variables range over arbitrary sets. 
\subsection{Summary of Results} \label{summary:sec}

Our  first basic result is a necessary and sufficient condition for a set of feasible voting patterns to be a possibility domain.

\begin{theorem}\label{basicthm} Let $X$ be a set of feasible voting patterns. Then the following statements are equivalent.
\begin{enumerate}
\item  $X$ is a possibility domain.
 \item $X$  admits a majority aggregator or it admits a minority aggregator or it has a  	non-trivial (non-dictatorial) binary aggregator.
     \end{enumerate}

\end{theorem}

  Clearly, in  Theorem \ref{basicthm}, only the direction $1 \Longrightarrow 2$ requires proof. (proof is given in Section \ref{sec:proofsA})  \remove{
  thus,  if instead of the notions of admitting a majority (minority) operation,  we adopted  weaker ones in which, for every two-element set $B_j$, we required the existence of an $\bar{f} = (f_1, \ldots, f_m)$ (that depends on $B_j$, as  in Bulatov's criterion), so that $f_j{\restriction{B_j}} = {\rm maj}$ (respectively, $f_j{\restriction{B_j}} = \oplus$), then   the theorem would still hold.}

An immediate corollary of Theorem  \ref{basicthm} is the following result, which although can also be immediately deduced from the  combination of the work of Dokow \& Holzman on one hand and Szegedy \& Xu on the other (see \cite[Theorem 2]{dokow2009aggregation} and \cite[Theorem 8]{szegedy2015impossibility}),  was not previously  explicitly mentioned.
\begin{corollary}
\label{cor:3_agg}
Let $X$ be a set of feasible voting patterns. Then the following statements are equivalent.
\begin{enumerate}
  \item $X$ is a possibility domain.
  \item $X$ has a  non-trivial (non-dictatorial) ternary aggregator.
\end{enumerate}
\end{corollary}
\begin{proof}
Indeed a binary aggregator 	can also be considered as a ternary one, by just ignoring one of its arguments.
\end{proof}

Note that, Corollary \ref{cor:3_agg}, in contrast to Theorem \ref{basicthm}, does not give information about the nature of the components of the ternary aggregators entailed when they are restricted to two-element sets of the corresponding set of votes, a necessary information in order to relate results in aggregation theory with complexity theoretic results (see Section \ref{upd:sect}). Observe however that for an arbitrary (supportive) binary aggregator $\bar{f}= (f_1, \ldots, f_m)$, every component $f_j$,   when restricted to a two-element subset $B_j \subseteq X_j$ considered as $\{0,1\}$,  is necessarily a projection function or $\land$ or $\lor$. So the information we mentioned before is  given {\em gratis} for binary aggregators.

  Also implicit  in Dokow and Holzman \cite{dokow2010aggregation} is the following result about the Boolean framework; we will provide an independent proof here (Section \ref{sec:proofsA}).

\begin{theorem}[Dokow and Holzman]\label{Bbasicthm}
Let $X$ be a set of feasible voting patterns in the Boolean framework. Then the following statements are equivalent.
\begin{enumerate}
  \item $X$ is a possibility domain.
  \item $X$ is  affine  or $X$  has a  	non-trivial (non-dictatorial) binary aggregator.	
\end{enumerate}
\end{theorem}


\smallskip

  Note that Dokow and Holzman \cite{dokow2010aggregation, dokow2010aggregationnonB} as well as  Szegedy and Xu  \cite{szegedy2015impossibility} used the concept of a {\em totally blocked} set of feasible voting patterns, a notion involving the propagation of individuals' positions,  to characterize when such a set is an impossibility domain (as mentioned earlier, the precise definition of this concept will be given in Section  \ref{totbl}).

We prove here that total-blockedness is simply a weak form of  impossibility domain restricted to binary aggregators: 
\begin{theorem} \label{thm:tot-block}
Let $X$ be a set of feasible voting patterns. Then the following statements are equivalent.
\begin{enumerate}
  \item $X$ is totally blocked.
  \item $X$ has no  non-trivial (non-dictatorial) binary aggregator.
\end{enumerate}
\end{theorem}
Theorem \ref{thm:tot-block}, which in some sense shows that the introduction of total blockedness can be avoided,  will be proved in Section \ref{totbl}.

Finally, in Section \ref{upd:sect}, we connect aggregation theory with complexity theory. Towards this we introduce the following stronger notion of  non-dictatorial aggregator.
\begin{definition}\label{def:upd}
An aggregator of a set of feasible voting patterns $X$ is called {\em {\un}  non-dictatorial} if for all $j=1, \ldots,m$ and every two-element subset $B_j \subseteq X_j$, $f_j{\restriction{B_j}}$ is not a projection function. 	$X$ is called {\em {\un} possibility domain} if it admits a uniformly non-dictatorial aggregator of some arity.
\end{definition}
\begin{example}
Let $W$ be the 1-in-3 relation defined in Example \ref{ex:1-in-3}. Then $W \times W$ is a possibility domain, which however is not a {\un} possibility domain in the sense of Definition \ref{def:upd}.
\end{example} 
Indeed, since $W$ is an impossibility domain, it easily follows that for every $n$, all $n$-ary aggregators of $W \times W$ are of the form 
\[(pr_d^n,pr_d^n,pr_d^n,pr_{d'}^n,pr_{d'}^n,pr_{d'}^n),\]  for  $d,d' \in \n$. \hfill $\Box$

Also,  obviously, every $X$ that admits a majority or a minority aggregator is a {\un} possibility domain (\upd).

Let now  $B$ be an arbitrary two-element set taken as  $\{0,1\}$ and consider the  binary logical operations  $\land$ and $\lor$ on $B$ (since we will always deal with both these logical operations concurrently, it does not matter which element of $B$ we take as 0 and which as 1). For notational convenience we define two {\em ternary}  operations on $B$ by:
\[\land^{(3)}(x,y,z)  = x \land y \land z \text{ and }   \lor^{(3)}(x,y,z) = x\lor y \lor z. \]

In Section  \ref{upd:sect} we will prove:
\begin{theorem}\label{thm:updcar}
Let $X$ be a set of feasible voting patterns. The following are equivalent:
\begin{enumerate}
	\item \label{item:one} $X$ is a {\un}  possibility domain.
	\item \label{item:two} For every $j=1, \ldots , m$ and for every two-element subset $B_j \subseteq X_j$, there is an aggregator $\bar{f} = (f_1, \ldots, f_m)$ (depending on $j$ and $B_j$) of some arity such that $f_j{\restriction{B_j}}$ is not a projection function.
	\item \label{item:three} There is ternary aggregator $\bar{f} = (f_1, \ldots, f_m)$ such that for all $j=1, \ldots, m$ and all two-element subsets $B_j \subseteq X_j$, $f_j{\restriction{B_j}} \in \{\land^{(3)}, \lor^{(3)}, {\rm maj}, \oplus\}$   (to which of the four  ternary operations $\land^{(3)}, \lor^{(3)}, {\rm maj}\text{ and }  \oplus$ the restriction $f_j{\restriction{B_j}}$ is equal to depends on $j$ and $B_j$).
	\item \label{item:four} There is ternary aggregator $\bar{f} = (f_1, \ldots, f_m)$ such that for all $j=1, \ldots, m$ and all $x,y \in X_j, f_j(x,y,y) = f_j(y,x,y) = f_j(y,y,x).$
	 \end{enumerate}
\end{theorem}
See also the related  result by Bulatov on ``three basic operations" \cite[Proposition 3.1]{bulatov2011complexity}, \cite[Proposition 2.2]{bulatov2016conservative} (that result however entails only  operations of arity two or three). Some techniques of the proof of Theorem \ref{thm:updcar}  had been used in the aforementioned works by Bulatov (see Section \ref{upd:sect}).\footnote{This came to the attention of the authors only after the present work was  essentially completed.}

\begin{corollary}\label{cor:cart_upd}
Let $X$ be a cartesian product $X= Y \times Z$, where $Y\subseteq \prod_{j=1}^l A_j$ and $Z \subseteq \prod_{j=l+1}^{m} A_j$, with $1 \leq l <m$. Assume that $Y,Z$ are \upd's. Then so is $X$. \end{corollary}
\begin{proof}
As a first case,    let $j =1, \ldots, l$. Also  let $B_j \subseteq X_j$ be a two-element set. Then there is an $n$-ary
aggregator $\bar{g} = (g_1, \ldots, g_l)$ such that $g_j{\restriction}{B_j}$ is not a projection.  Define $\bar{f}$ to be $(g_1,\ldots,  g_l, {\rm pr}^n_1, \ldots,  {\rm pr}^n_1)$, with $m-l$ copies of ${\rm pr}^n_1$ appended to $\bar{g}$. Then, since $j\leq l$,  $f_j{\restriction}{B_j} = g_j{\restriction}{B_j}$  is not a projection function. Similarly in the case $j=l+1, \ldots, m$. The result follows by item \eqref{item:two} of Theorem~\ref{thm:updcar}. \end{proof}

To state our result that connects the property of $X$ being a   {\un} possibility domain  with the property of tractability of a multi-sorted constraint satisfaction problems,    
we first introduce   some notions following closely \cite{bulatov2003algebraic, bulatov2011complexity}.

Again we consider a fixed set  $I = \{1, ,\ldots,m\}$, this time representing   {\em sorts},
and a family of finite sets 
 $\mathcal{A} = 
 \{A_1, \ldots, A_m\}$,  each of cardinality at least 2, representing the values the corresponding sorts can take. 
For any list of indices $(i_1, \ldots, i_k) \in I$ (not necessarily distinct),  a subset $R$ of $A_{i_1}\times \cdots \times A_{i_k}$, together with the list $(i_1, \ldots, i_k)$, will be called a multi-sorted relation over $\mathcal{A}$ with arity k and signature $(i_1, \ldots,  i_k)$. For any such relation $R$, the signature of $R$ will be denoted $\sigma(R)$. Any set of multi-sorted relations over $\mathcal{A}$ is called a multi-sorted constraint language over $\mathcal{A}$.
\begin{definition}[Multi-sorted CSP]\label{def:mcsp} Let $\Gamma$  be a multi-sorted constraint language over a collection of sets
$\mathcal{A} = (A_1, \ldots, A_m)$.
The multi-sorted constraint satisfaction
problem $\text{MCSP}(\Gamma)$ 
is defined to be the decision problem with instance $(V , \mathcal{A}, \delta, \mathcal{C})$, where $V$ is a finite set of variables; $\delta$  is a mapping from $V$ to $I$, called the sort-assignment function ($v$ belongs to the sort $\delta(v)$); $\mathcal{C}$ is a set of constraints where each constraint $C \in \mathcal{C}$  is a pair $(s,R)$, such that $s = (v_1 , . . . , v_k )$ is a tuple of variables of length $k$, called the constraint scope; R is an $k$-ary multi-sorted relation over $\mathcal{A}$  with signature $(\delta(v_1), . . . , \delta(v_k))$, called the constraint relation. 
The question is whether there exists a value-assignment, i.e. a mapping $\phi: V \mapsto \bigcup_{i=1}^m A_i$, such that, for each variable $v \in V, \phi(v) \in A_{\delta(v)}$, and for each constraint $(s, R) \in \mathcal{C}$, with $s = (v_1,...,v_k)$, the tuple $(\phi(v_1),...,\phi(v_k))$ belongs to $R$.
\end{definition}

A multi-sorted constraint language $\Gamma$ over $\mathcal{A}$ is said to be conservative  if for any $A \in \mathcal{A}$ and any $B \subseteq A$,   $B \in \Gamma$ (as a relation over A).

If $ X \subseteq \prod_{j=1}^m A_j $ is a set of feasible voting patterns, $X$ can be considered as multi-sorted relation with signature $(1, \ldots, m)$ (one sort for each issue).
Let $\Gamma^{\rm cons}_X$ be the multi-sorted conservative constraint language comprised of $X$ and all subsets of all $A_j, j=1, \ldots, m$, the latter  considered as relations over $A_j$.  

If all $A_j$ are equal and $|I|=1$, i.e. if there is no discrimination between sorts, then $\text{MCSP}(\Gamma)$ is denoted just by $\text{CSP}(\Gamma)$. 
If the sets of  votes for all issues are equal, then it is possible to consider a feasible set of votes $X$ as a one-sorted relation (all issues are of the same sort). In this framework, and in case all $A_j$ are equal to $\{0,1\}$, $\text{CSP}(\Gamma_X^{\rm cons})$ coincides with the problem   Schaefer \cite{schaefer1978complexity} introduced and called  the ``generalized satisfiability problem with constants", in his notation ${\rm SAT}_{\rm C}(\{X\})$ (the presence of the sets $\{0\}$ and $\{1\}$ in the constraint language amounts to allowing  constants, besides variables,   in the constraints).

We will prove (Section \ref{upd:sect}) that:
\begin{theorem}\label{thm:upd}
If $X$ is a uniformly possibility domain then $\text{MCSP}(\Gamma^{\rm cons}_X)$ is tractable; otherwise it is NP-complete.
\end{theorem}
Theorem \ref{thm:upd} is a corollary of  Theorem \ref{thm:updcar}  and Bulatov's dichotomy theorem for conservative multi-sorted constraint languages \cite[Theorem 2.16]{bulatov2011complexity}. For details see Section \ref{upd:sect}.

\begin{example}
Let $Y = \{0,1\}^3 \setminus \{(1,1,0)\}$ and let 
$$Z= \{(1,1,0), (0,1, 1), (1,0,1), (0,0,0)\}.$$ Then $Y, Z$ and hence, by Corollary \ref{cor:cart_upd}, $X= Y \times Z$  are \upd's, and therefore, by Theorem \ref{thm:upd},  $\text{MCSP}(\Gamma^{\rm cons}_X)$ is tractable. However,  the generalized satisfiability problem with constants  $\text{SAT}_{\text{C}}(\{X\})$ (equivalently $\text{CSP}(\Gamma^{\rm cons}_X)$) is NP-complete. 
\end{example}
Indeed, in Schaefer's \cite{schaefer1978complexity}	terminology, $Y$ is Horn, equivalently co-ordinate-wise closed under $\land$,  but is not dual Horn (equivalently not co-ordinate-wise closed under $\lor$), nor affine (equivalently, does not admit a minority aggregator) nor bijunctive (equivalently, does not admit a majority aggregator). Therefore, by co-ordinate-wise closureness under $\land$, $Y$ is \upd. Also, $Z$ is affine, but not Horn, nor dual Horn neither bijunctive. So, being affine,  $Z$ is \upd. The NP-completeness of $\text{SAT}_{\text{C}}(\{X\})$ (equivalently, the NP-completeness of $\text{CSP}(\Gamma^{\rm cons}_X)$) follows from Schaefer's archetypal dichotomy theorem \cite{schaefer1978complexity}, because $X$ is not Horn, dual Horn, affine, nor bijunctive.

\paragraph{The meta-problems} Given a family 
$\mathcal{A} = \{ A_1,  \ldots, A_m \}$ 
and a subset $X \subseteq \prod_{j=1}^m A_j $ as input, adopting a terminology used in complexity theory, we call  {\em meta-problems} the questions of (i)  whether $X$ is a possibility domain and (ii) whether $X$ is a {\un} possibility domain. 

 Theorem \ref{basicthm}, or even by Corollary \ref{cor:3_agg}, (respectively, Theorem \ref{thm:updcar})
  easily implies that  meta-problem (i) (respectively, meta-problem (ii)) is in NP. Indeed  we only have to guess  suitable ternary or binary operations and check for closureness. However, even if the sizes of all $A_j$'s are bounded by a constant (but $m$, the number of issues/sorts, is unbounded), then the problems could be NP-hard as there are exponentially large number of ternary or binary aggregators. The question of exactly bounding  this complexity is  the object of ongoing research. Of course, if besides the cardinality of all $A_j$, also their number $m$ is bounded, then Theorem~\ref{basicthm}, or even  Corollary~\ref{cor:3_agg},  (respectively, Theorem~\ref{thm:updcar}) implies that   meta-problem~(i) (respectively, meta-problem~(ii)) is tractable (for the first meta-problem, this was essentially observed by Szegedy and Xu \cite{szegedy2015impossibility}).
 
\section{Proofs of Theorems \ref{basicthm}  and \ref{Bbasicthm}} \label{proof:sect} \label{sec:proofsA}
Before proving Theorem \ref{basicthm}, we  introduce and illustrate a new notion, and establish two lemmas.

 Let $X$ be a set of feasible voting patterns and let $\bar{f} = (f_1, \ldots, f_m)$ be an $n$-ary aggregator for $X$.  
 \begin{definition}\label{def:locmon}
  We say that $\bar{f}$ is {\em locally monomorphic} if for all indices $i$ and $j$ with
$1\leq i, j\leq m$, for all   two-element subsets $B_i \subseteq X_i$ and $ B_j \subseteq X_j$,   for every bijection   $g: B_i \mapsto B_j$, and for all column vectors $x_i =(x_i^1, \ldots, x_i^n) \in B_i^n$, we have that
\[f_j(g(x_i^1), \ldots,g(x_i^n)) = g( f_i(x_i^1, \ldots, x_i^n)). \]
\end{definition}
Intuitively, the above definition says that, no matter how we identify the two elements of  $B_i$ and $B_j$ with 0 and 1,   the restrictions $f_i{\restriction{B_i}}$ and $f_j{\restriction{B_j}}$   are  equal as functions.

Notice that in the definition we are allowed to  have $i=j$, which implies that
 if in  a specific $B_j$ we interchange  the values  0 and 1 in the arguments of $f_j{\restriction{B_j}}$, then the bit that gives the image of $f_j{\restriction{B_j}}$ is  flipped.

 It follows immediately from the definitions that if an aggregator is trivial (dictatorial), then it is locally monomorphic. As we shall see next, different types of examples of locally monomorphic aggregators arise in  locally affine and locally bijunctive sets of feasible voting patterns.

 \begin{example} \label{loc-unim1-exam}
  Let $X$ be a  set of feasible voting patterns that admits a ternary aggregator  $\bar{f}=(f_1,...,f_m)$ that is either a majority or a minority operation. Then $\bar{f}=(f_1,...,f_m)$ is locally monomorphic.
\end{example}

 Suppose that $\bar{f}=(f_1,...,f_m)$ is a minority operation, i.e.  for every $j$ with $1\leq j\leq m$ and every two-element set $B_j\subseteq X_j$,
we have that $f_j \restriction_{B_j}= \oplus$.
 Let $i$, $j$ be such that $1\leq i, j, \leq m$, let $B_i =\{a,b\}  \subseteq X_i$, and let $B_j = \{c,d\}  \subseteq X_j$  (we make no assumption for the relation,  if any, between $a,b,c,d$).
 There are exactly two bijections  $g$ and $g'$  from $B_i$ to $B_j$, namely,
$$g(a)=c \ \& \ g(b)=d$$
$$g'(a)=d \ \& \ g'(b)=c$$
Suppose $(x,y,z)$ is the $i$-th column vector of $X^3$, with $x,y,z \in B_i$. Since $|B_i|=|B_j|=2$, it holds that $f_i \restriction_{B_i}=  \oplus = f_j \restriction_{B_j}$.
Without loss of generality, suppose $x=a, y=z=b$. Then
\begin{equation*}
\begin{split}
f_j(g(x),g(y),g(z)) & =f_j(c,d,d) \\
                                            & =\oplus(c,d,d)=c \\
                                            & =g(a)= g(\oplus(a,b,b)) \\
																						& =g(f_j(x,y,z)).
\end{split}
\end{equation*}
An analogous statement holds for $g'$. Since $i,j$ where arbitrary, we conclude that $\bar{f}$ is locally monomorphic. 

The proof for the case when $\bar{f}$ is a majority operation is similar.
\hfill  $\Box$

We  now  present the first lemma needed in the proof of Theorem \ref{basicthm}.

 \begin{lemma}\label{univlem} Let $X$ be a set of feasible voting patterns. If every binary aggregator for $X$ is trivial (dictatorial) on $X$, then, for every $n\geq 2$, every $n$-ary aggregator for $X$ is locally monomorphic.
 \end{lemma}
 \begin{proof}
 	Under the hypothesis that all binary aggregators are trivial, the conclusion is obviously true for binary aggregators.  By induction, suppose that the conclusion is true for all $(n-1)$-ary aggregators, where $n \geq 3$. Consider an $n$-ary aggregator $\bar{f} = (f_1, \ldots, f_m)$  and a pair $(B_i,B_j)$ of two-element subsets $B_i \subseteq X_i$ and $B_j \subseteq X_j$. To render the notation less cumbersome, we will take the liberty to  denote the two elements  of both $B_i$ and $B_j$  as $0$ and $1$. Assume now, towards a contradiction, that  there are  a  column-vector $(a^1, \ldots, a^n)$ with $a^i \in \{0,1\}$, $1\leq i\leq n$,    a ``copy" of this vector belonging  to $B_i^n$,   another copy belonging to  $B_j^n$, such that  $f_i(a^1, \ldots, a^n) \neq f_j(a^1, \ldots, a^n)$. Since $n\geq 3$, by the pigeonhole principle applied to two holes and at least three pigeons, there is a pair of  coordinates of $(a^1, \ldots, a^n)$ that coincide. Without loss of generality, assume  that these two coordinates are the two last ones, i.e., $a^{n-1} = a^n$. We  now define an $(n-1)$-ary aggregator $\bar{g} =(g_1, \ldots, g_m)$ as follows: given $n-1$ voting patterns $(x_1^i, \ldots, x_m^i)$, $i=1, \ldots, n-1$, define
 	$n$
 	voting patterns by
 	just repeating the last one
 	and then for all $j=1, \ldots, m$, define
 	\[g_j(x_j^1, \ldots, x_j^{n-1}) =  f_j(x_j^1, \ldots, x_j^{n-1}, x_j^{n-1}).\]
 	It is straightforward to verify that $\bar{g}$ is an $(n-1)$-ary aggregator on $X$ that is not locally monomorphic, which contradicts the inductive hypothesis.  \end{proof}

 \begin{remark}{\em
 The preceding argument  generalizes to  arbitrary cardinalities in the following way:  if every aggregator of arity at most $s$ on $X$  is trivial, then every aggregator on $X$ is {\em $s$-locally monomorphic}, meaning that for every  $k\leq s$ and for all  sets $B_j\subseteq X_j$ of cardinality  $k$,  the functions $f_j{\restriction{B_j}}$ are all equal up to bijections between the  $B_j$'s.}
\end{remark}

We continue with a technical lemma whose proof was  inspired by a proof in Dokow and Holzman \cite[Proposition 5]{dokow2010aggregationnonB}.

\begin{lemma}\label{lem:technic} Assume that for all integers $n\geq 2$ and for every $n$-ary aggregator $\bar{f} =(f_1, \ldots, f_m)$, there is an integer $d \leq n$ such that for  every integer $j\leq  m$ and every two-element subset $B_j \subseteq X_j$,  the restriction $f_j{\restriction{B_j}}$ is equal to ${\rm pr}^n_d$, the $n$-ary projection on the $d$-th coordinate.	 Then
	for all integers $n\geq 2$ and for every $n$-ary aggregator $\bar{f} =(f_1, \ldots, f_m)$ and for all $s \geq 2$,  there is an integer $d \leq n$ such that for  every integer $j\leq  m$ and every subset $B_j \subseteq X_j$ of cardinality at most $s$,  the restriction $f_j{\restriction{B_j}}$ is equal to ${\rm pr}^n_d$.
\end{lemma}
\begin{proof}
The proof will be given by induction on $s$. The induction basis $s=2$ is given by hypothesis.	
Before delving into the inductive step of the proof  and for the purpose of making the intuition behind it clearer, let us mention the following fact whose proof is left to the reader. This fact illustrates the idea for obtaining a non-trivial aggregator of lower arity from one of higher arity.

\smallskip

{\bf Fact.} Let $A$ be a set and let $f: A^3  \mapsto  A$ be a supportive  function
such that if among $x_1, x_2, x_3$ at most two are different, then $f(x_1, x_2, x_3)=x_1$. Assume also that there exist pairwise distinct $a_1, a_2, a_3$ such that $f(a_1, a_2, a_3) = a_2$; in the terminology of universal algebra, f is a {\em semi-projection}, but not a projection. Define $g(x_1, x_2) = f(x_1, f(x_1, x_2, a_3), a_3)$. Then, by distinguishing cases as to the value of $f(x_1,x_2,a_3)$, it is easy to verify that $g$ is  supportive; however, $g$ is not a projection function because  $g(a_1, a_2) = a_2$,  whereas $g(a_1, a_3) = a_1$.

\smallskip

We continue with the inductive step of the proof of Lemma \ref{lem:technic}.    Assume that the claim holds for $s-1$. We may assume that $s \leq n$, lest the induction hypothesis applies.  Let $\bar{f}$ be an $n$-ary aggregator for $X$ and let $d$ be the integer obtained by applying the induction hypothesis on $s-1$. Assume, without loss of generality that $d=1$.  Then for every $j\leq m$ and for every  subset $B_j \subseteq X_j$ of cardinality at most $s-1$, we have that   $f_{j}{\restriction{B_j}} ={\rm pr}^n_1$, the $n$-ary projection function on $d=1$.
We will show that the same
holds for subsets $B_j\subseteq X_j $ of cardinality at most $s$.

Assume towards a contradiction that there exists an integer $j_0\leq m $  and row vectors ${a}^1, \ldots, {a}^n$ in $X$  such that the set
$B_{j_0}= \{a_{j_0}^1, \ldots, a_{j_0}^n\}$
 has  cardinality $s$ and
 \begin{equation} \label{i0-}f_{j_0}(a_{j_0}^1,  \ldots, a_{j_0}^n) \neq a_{j_0}^1.\end{equation}
 By  supportiveness, there exists $ i_0 \in \{2, ,\ldots, n\}$ such that
  \begin{equation}\label{i0}f_{j_0}(a_{j_0}^1,  \ldots, a_{j_0}^n) = a_{j_0}^{i_0}.\end{equation}
 Let $\{k_1, \ldots, k_s\}$ be a subset of $\{1, \ldots, n\}$ of  maximum cardinality  such that the $a_{j_0}^{k_1}, \ldots, a_{j_0}^{k_s}$ are pairwise distinct (the maximum such cardinality could not be less than $s$ because of the induction hypothesis).   Obviously, if $i \not \in \{k_1, \ldots, k_s\}$, then there exists
$l \in \{1, \ldots, s\}$ such that $a_{j_0}^i = a_{j_0}^{k_l}$. So, we may assume that the $i_0$ in equation \eqref{i0} above belongs to $\{k_1, \ldots, k_s\}$ (and is different from 1, because by equation \eqref{i0-}, $a_{j_0}^{i_0} \neq a_{j_0}^1$). Since $s\geq 3$ there is an element in $\{k_1, \ldots, k_s\}$ different from both $1$ and $i_0$. Assume, without loss of generality, that this element is $k_s$. Let $B_{j_0}^- = \{a_{j_0}^{k_1}, \ldots, a_{j_0}^{k_{s-1}}\}$. We define an $(s-1)$-ary aggregator $\scaledbar{f}^- =(f_1^-, \ldots, f_m^-)$ as follows:
\[f_j^- (x^1_j, \ldots, x^{s-1}_j) = f_j(y^1_j, \ldots, y^n_j), \  j =1, \ldots, m,\]
where,
\begin{equation}
y^i_j = \begin{cases}x^l_j & \text{ if }  i= k_l \text{ for some  } l= 1, \ldots, {s-1},\\ x^l_j & \text{ if } i \not\in \{k_1, \ldots, k_s\} \text{ and } a_{j_0}^i = a_{j_0}^{k_l}\text{ for some  } l= 1, \ldots, {s-1},\\
a^{k_s}_j  & \text{ if } i=k_s, \\ a^{k_s}_j  &  \text { if } i \not\in \{k_1, \ldots, k_s\} \text{ and } a_{j_0}^i = a_{j_0}^{k_s}.
 	\end{cases}
	\end{equation}
Intuitively, to compute $f_j^- (x^1_j, \ldots, x^{s-1}_j)$, we put $x^1_j, \ldots, x^{s-1}_j$ as arguments of $f_j$  at the places $\{k^1,\ldots, k^{s-1}\}$,  we  put  $a^{k_s}_j$ as argument of $f_j$ at the place  $k_s$, and finally, as arguments of $f_j$ at places not in $\{k^1,\ldots, k^s\}$, we put either a copy of one $\{x^1, \ldots, x^{s-1}\}$ or a copy of $a^{k_s}$ following the pattern by which in the vector $a^1_{j_0}, \ldots, a^n_{j_0}$ the coordinates $a^{k_1}_{j_0}, \ldots, a^{k_s}_{j_0}$ repeat themselves in the places not in $\{k^1,\ldots, k^s\}$.

First observe that $\scaledbar{f}^-$ is supportive. Indeed this follows from the observation that $f_j^-$  can never take the value $a_j^{k_s}$.
Then observe that $\scaledbar{f}^- =(f_1^-, \ldots, f_m^-)$ is  an aggregator on $X$, because all  row vectors $y^1, \ldots, y^n$ defined above  belong to $X$ (each  is either some $x^i$ or some $a^i$).

 It is obvious that \[f_{j_0}^-(a_{j_0}^{k_1},  \ldots, a_{j_0}^{k_{s-1}}) = f_{j_0}(a_{j_0}^1, \ldots, a_{j_0}^n) = a_{j_0}^{i_0}.\]
Also, by the inductive hypothesis, if not all $x^1_{j_0},\ldots,  x^{s-1}_{j_0}$ are equal then \[f^-_{j_0} (x^1_{j_0}, \ldots, x^{s-1}_{j_0}) = x^1_{j_0}.\] Therefore, $f_{j_0}^-{\restriction{B_{j_0}^-}} \neq {\rm pr}^{s-1}_1$, which is contradiction that concludes the proof of Lemma \ref{lem:technic}. \end{proof}

Next, we bring into the picture some basic concepts and results from universal algebra; we refer the reader
to Szendrei's monograph  \cite{szendrei1986clones}) for additional information and background.
 A {\em clone}  on a finite  set $A$ is  a set $\mathcal C$ of finitary operations on $A$ (i.e., functions from a power of $A$ to $A$) such that $\mathcal C$  contains all projection functions and is closed under arbitrary compositions (superpositions).  The proof of the next lemma is straightforward, and we omit it.

\begin{lemma}\label{lem:clone} Let $X$ be a set of feasible voting patterns. For every $j$ with $1\leq j\leq m$ and every subset
$B_j\subseteq X_j$, the set ${\mathcal C}_{B_j}$ of the restrictions
$f_j{\restriction{B_j}}$ of the $j$-th components of aggregators $\bar{f} = (f_1, \ldots, f_m)$ for $X$ is a clone on $B_j$.
\end{lemma}

Post \cite{post1941two} classified all clones on a two-element set (for more recent expositions of Post's pioneering results, see, e.g.,
\cite{szendrei1986clones} or  \cite{pelletier1990post}). One of Post's main findings is that if $\mathcal C$ is
  a clone of conservative functions on a two-element set, then either $\mathcal C$ contains  only projection functions or $\mathcal C$ contains one of the following operations: the binary operation $\land$, the binary operation $\lor$, the ternary operation $\oplus$, the ternary operation  ${\rm maj}$.

We now restate  Theorem \ref{basicthm} and prove it.

\setcounter{theorem}{0}

\begin{theorem}\label{basicthm} Let $X$ be a set of feasible voting patterns. Then the following statements are equivalent.
\begin{enumerate}
\item  $X$ is a possibility domain.
 \item $X$  admits a majority aggregator or it admits a minority aggregator or it has a  	non-trivial (non-dictatorial) binary aggregator.
     \end{enumerate}
\end{theorem}
\begin{proof}
  As stated earlier, only the direction $1 \Longrightarrow 2$ requires proof. In the contrapositive, we will only prove that if $X$ does not admit a majority, neither a minority, nor a  non-trivial binary aggregator, then $X$ does not have an $n$-ary non-trivial aggregator for any $n$.
Towards this goal, and assuming that  $X$ is as stated, we will first show that the hypothesis of Lemma \ref{lem:technic} holds. Then the required  follows from the conclusion of Lemma \ref{lem:technic} by taking $s= \max\{|X_j|: 1\leq j\leq m\}$.
\smallskip

Given $j\leq m$ and a two-element subset $B_j\subseteq X_j$, consider the clone ${\mathcal C}_{B_j}$. If ${\mathcal C}_{B_j}$  contained one of the the binary operations $\land$ or  $\lor$ then $X$ would have a binary non-trivial aggregator, a contradiction. If, on the other hand, ${\mathcal C}_{B_j}$ contained the ternary operation $\oplus$ or the ternary operation $\text{maj}$, then,  by Lemma \ref{univlem}, $X$ would admit a minority or a majority aggregator, a contradiction as well.  So, by the aforementioned Post's result, all elements of
${\mathcal C}_{B_j}$, no matter what their arity is,  are projection functions. By Lemma \ref{univlem} again, since $X$ has no binary non-trivial aggregator, we have that for every $n$ and for every $n$-ary aggregator $\bar{f} =(f_1, \ldots, f_m)$, there exists an integer $d \leq  n$ such that for every $j \leq m$ and every  two-element set $B_j \subseteq X_j $, the restriction $f_j{\restriction{B_j}}$ is equal to ${\rm pr}^n_d$, the $n$-ary projection on the $d$-th coordinate. This concludes the proof of  Theorem \ref{basicthm}.
\end{proof}

We conclude this section by restating Theorem \ref{Bbasicthm} and proving it; as mentioned earlier, this result is implicit in \cite{dokow2010aggregation}. Recall that in the Boolean framework, $X$ admits a majority (minority)  aggregator iff a  logical relation $X$ is bijunctive (affine). See Schaefer \cite{schaefer1978complexity} for definitions.

\begin{theorem}[Dokow and Holzman]\label{Bbasicthm}
Let $X$ be a set of feasible voting patterns in the Boolean framework. Then the following statements are equivalent.
\begin{enumerate}
  \item $X$ is a possibility domain.
  \item $X$ is  affine  or $X$  has a  	non-trivial (non-dictatorial) binary aggregator.	
\end{enumerate}
\end{theorem}

\begin{proof}  Only the direction $1 \Longrightarrow 2$ requires proof.
Assume that $X$ is a possibility domain in the Boolean framework. By Theorem \ref{basicthm}, $X$ admits either a majority or a minority aggregator  or $X$ has non-trivial binary aggregator. Since we are in the Boolean framework, this means that $X$ is affine or $X$ is bijunctive or $X$ has a non-trivial binary aggregator. If $X$ has at most two elements, then $X$ is closed under $\oplus$, hence $X$ is affine. So, it suffices to show that if $X$ is bijunctive and has at least three elements, then $X$ has a non-trivial binary aggregator. In turn, this follows immediately from the following claim.

\smallskip

\begin{claim} \label{bijunctive-claim}
Let $X$ be a bijunctive relation on $\{0,1\}$ with at least three elements.
  If $X$ is not degenerate (i.e., every $X_j$ has at least two elements), 
then  $X$ has a binary non-monomorphic aggregator.	
\end{claim}

To prove Claim \ref{bijunctive-claim},  fix an element $\bar{a} = (a_1, \ldots, a_m) \in X$. Define the following binary aggregator, where $\bar{x} =
(x_1, \ldots, x_m)$ and $\bar{y} = (y_1, \ldots, y_m)$ are arbitrary elements of $X$:
\[ {\bar{f}}^{\scaledbar{a}}(\bar{x},\bar{y} ) = ({\rm maj}(x_1, y_1, a_1), \ldots, {\rm maj} (x_m, y_m, a_m)).
\]

First, observe that $\bar{f}^{\bar{a}}$ is indeed an aggregator for $X$.
Since 
$X$ is closed under
${\rm maj}$, all we have to prove is that $\bar{f}^{\bar{a}}$ is supportive. But this is obvious, because, for $j\leq m$, if $x_j = a_j$ or $y_j=a_j$, then ${\rm maj}(x_j, y_j, a_j)= x_j$ or ${\rm maj}(x_j, y_j, a_j)= y_j$. If $x_j\not =  a_j$ and $y_j \not = a_j$, then $x_j = y_j$, hence
 ${\rm maj}(x_j, y_j, a_j) = x_j=y_j.$

Now assuming that $X$ contains more than two elements and is not degenerate, we will show    that there exists a row vector $\bar{a} = (a_1, \ldots, a_m) \in X$
such that $\bar{f}^{\scaledbar{a}}$ is not monomorphic, i.e., there are distinct  $i,j=1, \ldots,m$ such that $\bar{f}^{\bar{a}}_i \neq \bar{f}^{\bar{a}}_j$, and thus the proof of the claim will be concluded.

Observe first that if for all distinct   $i\leq m$ and $j \leq m$ one of the following (depending on $i,j$) were true:
\begin{itemize}
\item for all vectors $\bar{u} \in X$, we have that $u_i = u_j$ or 
\item for all vectors $\bar{u} \in X$, we have that $u_i \neq  u_j$, 
 \end{itemize}
then it would follow  that there exist  only two elements in $X$ which at every coordinate have complementary values, contradicting the hypothesis that $X$ contains more than two elements. Therefore, there exist two distinct integers $i\leq m$ and $j\leq  m$ for which there are two elements  $\bar{u},\bar{v} \in X$ such that $u_i \neq u_j$ and $v_i = v_j$. Combining the last statement with the non-degeneracy of $X$, we conclude,   by an easy case analysis,  that there exist three elements $\bar{u},\bar{v}, \bar{w} \in X$ such that at least  one of the following four cases holds:
\begin{enumerate}[label =(\roman*)]
\item the $i$-th and $j$-th coordinates of $\bar{u},\bar{v}, \bar{w}$ are $(1,0),(0,1),(1,1)$, respectively, \label{casei}
\item  the $i$-th and $j$-th coordinates of $\bar{u},\bar{v}, \bar{w}$ are $(1,0),(0,1),(0,0)$, respectively, \label{caseii}
\item the $i$-th and $j$-th coordinates of $\bar{u},\bar{v}, \bar{w}$ are $(1,0),(1,1),(0,0)$, respectively, \label{caseiii}
\item the $i$-th and $j$-th coordinates of $\bar{u},\bar{v}, \bar{w}$ are $(1,0),(0,1),(0,0)$, respectively. \label{caseiv}
\end{enumerate}
In cases \ref{casei} and \ref{caseii}, by computing the $i$-th and $j$-th coordinates of ${\bar{f}}^{\bar{u}}(\bar{u},\bar{v} )$ and ${\bar{f}}^{\bar{u}}(\bar{v},\bar{u} )$, we conclude that ${\bar{f}}^{\bar{u}}_i = \lor$ and ${\bar{f}}^{\bar{u}}_j = \land $, so ${\bar{f}}^{\bar{u}}_i \neq {\bar{f}}^{\bar{u}}_j $.
In case \ref{caseiii}, by computing  the $i$-th and $j$-th coordinates of ${\bar{f}}^{\bar{u}}(\bar{v},\bar{w} )$, we conclude that  ${\bar{f}}^{\bar{u}}_i \neq {\bar{f}}^{\bar{u}}_j $. Case \ref{caseiv} is similar. This completes the proof of Claim \ref{bijunctive-claim}  and  of  Theorem \ref{Bbasicthm}. 
\end{proof}

\section{Total Blockedness}\label{totbl}
In this section,  we will follow more closely the notation in \cite{dokow2010aggregationnonB}. Let $X$ be a set of feasible voting patterns.

Given subsets  $B_j \subseteq X_j, j=1, \ldots, m$,
the product $B = \prod_{j=1}^m B_j $ is called a {\em sub-box}. It is called a {\em $2$-sub-box} if $|B_j| =2$ for all $j$.

Elements of a box $B$ that belong also to $X$  will  be called {\em feasible evaluations within $B$} (in the sense that each issue $j=1,
 \ldots, m$ is ``evaluated" within $B$). If $ K \subseteq  \{1, \ldots, m \}$, a vector
 $ x \in \prod_{j \in K}B_j$ is called a {\em partial feasible  evaluation within $B$} if there exists a feasible $y\in B$ that extends $x$, i.e. $x_j = y_j, \forall j \in K$, whereas $x$  is called an infeasible within $B$ partial evaluation otherwise; finally, $x$  is called a {\em $B$-Minimal Infeasible  Partial Evaluation} ($B$-MIPE) if $x$ is a  infeasible  within $B$ partial evaluation and if for every $j \in K$,  there is a $b_j \in B_j$ so that changing the  $j$-th  coordinate of $x$ to $b_j$ results in a feasible  partial evaluation within $B$.

 We define a directed graph  $G_X$ whose vertices correspond to all pairs of {\em distinct} elements $u,u'$ in $X_j$ for all $j =1, \ldots m$  and are to be denoted by $uu'_j$. Given two vertices $uu'_k, vv'_l$ with $k \neq l$, we connect them by a directed edge from  $uu'_j$ towards $vv'_l$ if there exists a 2-sub-box $B = \prod_{j=1}^m B_j $, a $K \subseteq \{1, \ldots, m\}$ and a $B$-MIPE $x = (x_j)_{j\in K}$ so that $k,l \in K$ and $B_k = \{u,u'\}$ and $ B_l =\{v,v'\}$ and $x_k =u$ and $x_l = v'$. We denote such a directed edge by $uu'_k \underset{B,x,K}{\longrightarrow} vv'_l$ (or just $uu'_k \rightarrow vv'_l$ in case $B,x, K$ are understood from the context).

 We say that $X$ is {\em totally blocked} if $G_X$ is strongly connected, i.e., for every two distinct vertices $uu'_k, vv'_l$  are connected by a directed path (this must hold even if $k=l$). This definition, given 
  by Dokow and Holzman \cite{dokow2010aggregationnonB},
  is  a generalization  to the case where $A_j$ are allowed to have arbitrary cardinalities of a corresponding definition for the Boolean framework (all $A_j$ of cardinality 2), originally given in \cite{nehring2002stategy}.

 The length of a $B$-MIPE $x = (x_j)_{j \in K}$ is $|K|$. We say that $X$ is {\em multiply constrained} (definition by Dokow and Holzman \cite{dokow2010aggregationnonB}) if there exists a sub-box $B$ for which there exists a $B$-MIPE of length at least 3.

 Dokow and Holzman \cite{dokow2010aggregationnonB} have established the following two results.

 \begin{customthm}{A}[Dokow and Holzman \cite{dokow2010aggregationnonB}] 	
 If $X$ is totally blocked and multiply constrained then $X$ is an impossibility domain.
 \end{customthm}
 
 \begin{customthm}{B}[Dokow and Holzman \cite{dokow2010aggregationnonB}]\label{D-H-B}
If $X$ is not totally blocked, then $X$ is a possibility domain; in fact for every $n\geq 2$ there is a non-trivial $n$-ary  aggregator.
  \end{customthm}

 Quite recently, Szegedy and Xu \cite{szegedy2015impossibility} provided a sufficient and necessary condition for $X$ to be an impossibility domain. Their result is as follows.
 
 \begin{customthm}{C}[Szegedy and Xu \cite{szegedy2015impossibility}] If $X$ is totally blocked  then  $X$ is an impossibility domain if and only if $X$  contains no binary and no   ternary non-trivial aggregator.
 	\end{customthm}
Observe that  latter two theorems obviously imply Corollary \ref{cor:3_agg}, however this was not previously explicitly stated.

We now prove that:  

\begin{theorem}\label{thm:tot-block}
Let $X$ be a set of feasible voting patterns. Then the following statements are equivalent.
\begin{enumerate}
  \item $X$ is totally blocked.
  \item $X$ has no  non-trivial (non-dictatorial) binary aggregator.
\end{enumerate}
  \end{theorem}

 Direction $2 \Longrightarrow 1$ 
 is contained in Dokow and Holzman \cite[Theorem 2]{dokow2010aggregationnonB} (Theorem~\ref{D-H-B} above);  for completeness, we give an independent proof of both directions.
	
\begin{proof} We start with 
direction $1 \Longrightarrow 2$.
Consider at first  two vertices $uu'_k, vv'_l$ of $G_X$ (with $k \neq l$) connected by an edge $uu'_k\rightarrow vv'_l$. Then there exists a 2-sub-box $B = \prod_{j=1}^m B_j$ with $B_k = \{u, u'\}$ and $B_l = \{v, v'\}$ and a $B$-MIPE $x= (x_j)_{j \in K}$ such that  $\{k,l\} \subseteq K$  and $x_k =u, x_l=v'$.

\begin{claim} \label{last-claim} For every binary aggregator
$\bar{f} = (f_1, \ldots, f_m)$ of  $X$, if $f_k(u,u') =u$ then $f_l(v,v') =v$. \end{claim}

\medskip\noindent {\em Proof of Claim.} By the minimality of $x$ within $B$ if we flip $x_k$ from $u$ to $u'$ or if we flip $x_l$ from $v'$ to $v$, then we get, in both cases,  respective feasible  evaluations within $B$. Therefore,  there are two total evaluations $e$ and $e'$ in $X\cap B$ such that
\begin{itemize}
	 \item $e_k = u'$ and \item $e_s = x_s$ for $s\in K, s \neq k$ (in particular $e_l = v'$), \end{itemize} because we get a feasible evaluation within $B$ by flipping $x_k $ from $u$ to $u'$ and
 \begin{itemize}
 \item $e'_l = v$ and \item $e'_s = x_s$ for $s\in K, s \neq l$ (in particular $e'_k = u$).
 \end{itemize}
This is so because we get a feasible evaluation within $B$ by flipping $x_l$ from  $v'$ to $v$.

If we assume, towards a contradiction,  that $f_k(u,\bar{u}) =u$ and  $f_l(v,\bar{v}) =\bar{v}$,  we immediately have that the evaluation \[\bar{f}(e, e') := (f_1(e_1, e'_1), \ldots, f_m(e_m, e'_m))\] extends $(x_j)_{j \in K}$, contradicting the latter's infeasibility within $B$. This completes the proof of Claim \ref{last-claim}. \hfill$\Box$

From  Claim \ref{last-claim},  we  get that if $u_k \rightarrow\rightarrow v_l$ and $f_k(u,\scaledbar{u}) =u$, then $f_l(v,\scaledbar{v}) =v$ (even if $k=l$). From this, it immediately follows that if $G_X$ is strongly connected, then every binary aggregator of $X$ is dictatorial.

We will now prove Direction $2 \Longrightarrow 1$ of Theorem \ref{thm:tot-block} , namely, that if $X$ is not totally blocked, then there is a  non-trivial binary aggregator (this part is contained in \cite[Theorem 2]{dokow2010aggregationnonB} --Theorem  \ref{D-H-B} above).
Since $G_X$ is not strongly connected, there is a partition of the vertices of $G_X$ into two mutually disjoint and non-empty subsets $V_1$ and $V_2$ so that there is no edge from a vertex of $V_1$ towards a vertex in $V_2$. We now define a  $\bar{f} = (f_1, \ldots, f_m)$, where $f_k: A_k^2 \mapsto A_k$,  as follows:
\begin{equation}\label{eqfk1}
	f_k( u,u') = \begin{cases} u &\mbox{ if } u,u' \in X_k \text{ and }  uu'_k \in V_1 \text{ and } u \neq u', \\   u' & \text{ if }  u,u' \in X_k \text{ and } \mbox{ if } uu'_k \in V_2 \text{ and } u \neq u',\\ u & \text { if }u=u' \text{ or } u  \in A_k \setminus X_k \text{ or } u' \in A_k \setminus X_k.  \end{cases}
\end{equation}
In other words, for two differing values $u$ and $u'$ in $X_k$, the function $f_k$ is defined as the projection on the first coordinate if $uu'_k \in V_1$, and as the projection onto the second coordinate if $uu'_k \in V_2$; we also define $f_k(u,u) = u$ if $u=u'$ or if  either  $u$ or $u'$ is not in $X_k$ (i.e., when at least one  of them is  not a projection onto the $k$-th coordinate  of an element of $X$, in this latter case the value of $f_k(u,u')$ can be arbitrarily defined, as it has no effect on the properties of $\bar{f}$).

 Notice that $\bar{f}$ is non-trivial, because $V_1$ and $V_2$ are not empty.

All that remains to be shown is that $X$ is closed under $\bar{f}$, i.e., if $e = (e_1, \ldots, e_m),e' =(e'_1, \ldots, e'_m) \in X$ are two total feasible evaluations, then \begin{equation}\label{toprove} \bar{f}(e, e')  := (f_1(e_1, e'_1), \ldots f_m(e_m, e'_m)) \in X.\end{equation}
Let \[L = \{j=1, \ldots, m \mid  e_j \neq e'_j\}.\] For an arbitrary $j \in L$,
   define  $\text{vertex}_j(e, e')$ to be the vertex $uu'_j$ of $G_X$, where   $u = e_j$ and $u' = e'_j$ .

If now $\bar{f}(e, e') = e$ or if $\bar{f}(e, e') = e'$, then obviously
\eqref{toprove} is satisfied.
So assume that \begin{equation}\bar{f}(e, e') \neq e \text{ and }  \bar{f}(e, e') \neq e'.\end{equation}
Also, towards showing \eqref{toprove} by contradiction, assume \begin{equation}\label{contr}\bar{f}(e, e') \not\in X. \end{equation}
Define now a 2-sub-box $B = (B_j)_{j=1, \ldots,m}$ as follows:

\begin{equation}\label{B} B_j = \begin{cases}\{e_j, e'_j\} & \text{ if } e_j \neq e'_j, \\ \{e_j, a_j\} & \text{ otherwise }, \end{cases} \end{equation}
where $a_j$ is an arbitrary element $\neq e_j$  of $X_j$ (the latter choice is only made to ensure that $|B_j| =2$ in all cases).

Because of \eqref{contr} and \eqref{B}, we have that
$\bar{f}(e, e')$ is a total evaluation infeasible within $B$.  Towards constructing a $B$-MIPE, delete one after the other (and as far as it can go) coordinates
of $\bar{f}(e, e')$, while taking care not to destroy infeasibility within $B$. Let $K \subseteq \{1, \ldots, m\}$ be the subset of coordinate indices that remain at the end of this process. Then the partial evaluation
\begin{equation}\label{x} x := \left(f_j(e_j, e'_j)\right)_{j \in K}\end{equation}
is infeasible within $B$. Therefore, lest
$e$ or $e'$ extends $x = \left(f_j(e_j, e'_j)\right)_{j \in K}$ (not permissible because the latter partial evaluation is infeasible),  there exist $k,l  \in K$ such that \begin{equation}\label{els}
 	e_k \neq e'_k \text{ and } e_l \neq e'_l \end{equation} and also \begin{equation}\label{ls} f_k(e_k, e'_k) = e_k \text{ and } f_l(e_l, e'_l) = e'_l. \end{equation}
 But then if we set \begin{equation}\label{ls2}u = e_k, u' = e'_k, v = e_l, v' = e'_l,\end{equation}
 we have, by \eqref{eqfk1}, \eqref{els}, \eqref{ls} and \eqref{ls2}, that \begin{equation}\label{last}\text{vertex}_k(e, e') = uu'_k \in V_1\text{ and }  \text{vertex}_l(e, e') = vv'_l \in V_2\end{equation}
 and, by \eqref{eqfk1}, \eqref{ls} and  \eqref{ls2},    we get that
 \[uu'_k \underset{B,x,K}{\longrightarrow} vv'_l\]
 which by \eqref{last} is a contradiction, because we get an edge from $V_1$ to $V_2$. This completes the proof of Theorem \ref{thm:tot-block}.
\end{proof}

\section{Uniformly Possibility Domains} \label{upd:sect}
We start with the following lemma:
\begin{lemma}[Superposition of aggregators]\label{lem:superposition} 
Let $\bar{f} = (f_1, \ldots, f_m)$ be an $n$-ary aggregator and let \[\overline{h^1} = (h^1_1, \ldots, h^1_m), \ldots, \overline{h^n} = (h^n_1, \ldots, h^n_m)\] be $n$ $k$-ary aggregators (all on $m$ issues). Then the $m$-tuple of $k$-ary functions 	$(g_1, \ldots, g_m)$ defined by:
\[g_j(x_1, \ldots, x_k) = f_j(h_j^1(x_1, \ldots, x_k), \ldots, h_j^n(x_1, \ldots, x_k)), j =1, \ldots, m\] is also an aggregator.
\end{lemma}
\begin{proof}
Let $x^l_j, l=1, \ldots, k, j=1, \ldots, m$ be a $k\times m$ matrix whose rows are in $X$. Since the $\overline{h^i}, i=1, \ldots, n$ are $k$-ary aggregators, we conclude that for all $i=1, \ldots, n$, \[(h_1^i(x^1_1, \ldots, x_1^k), \ldots, h_m^i(x_m^1, \ldots, x_m^k)) \in X.\]	
We now apply the aggregator $\bar{f} = (f_1, \ldots,  f_m)$ to the $n\times m$ matrix 
\[h^i_j(x_j^1, \ldots, x_j^k), i=1, \ldots, n, j=1, \ldots, m,\] 
which concludes the proof.
\end{proof}

Using the above lemma we will assume below, often tacitly,  that various tuples of functions obtained by superposition of aggregators with other aggregators,  like projections, are  aggregators as well. 

We now prove three lemmas: 
\begin{lemma}\label{lem:four_functions}
Let $A$ be an arbitrary set and $f: A^3 \mapsto A$ a ternary supportive operation on $A$, and $B$ a two-element subset of $A$ taken as $\{0,1\}$. Then $f{\restriction}{B}$ is commutative iff $f{\restriction}{B} \in \{ \land^{(3)},\lor^{(3)}, {\rm maj}, \oplus\}$.
	\end{lemma}
	\begin{proof}
	Only the  sufficiency of commutativity of $f{\restriction}{B}$ for its  being one of  $\land^{(3)},\lor^{(3)}$, ${\rm maj}, \oplus$ is not absolutely trivial.  Since $f$ is supportive, $f(0,0,0) =0$ and $f(1,1,1) =1$. Assume $f{\restriction}{\{0,1\}}$ is commutative. Let 
	\[f(1,0,0) = f(0,1,0) = f(0,0,1) := a, \text{ and }\] \[f(0,1,1) = f(1,0,1) = f(1,1,0) := b. \]
	By supportiveness, $a,b \in \{0,1\}$.	
	If $a=b=0$, then $f=\wedge^{(3)}$; if $a=b=1$, $f=\vee^{(3)}$; if $a=0$ and $b=1$, $f={\rm maj}$; and if $a=1$ and $b=0$, $f =\oplus$.
\end{proof}
\begin{lemma}\label{lem:commutative}
Let $A$ be an arbitrary set and $f, g: A^3 \mapsto A$ two ternary supportive operations on $A$. Define the supportive as well ternary operation \[h(x,y,z) = f(g(x,y,z), g(y,z,x), g(z,x,y)).\]
If $B$ is a two-element subset of $A$ then $h{\restriction}{B}$ is commutative if either $f{\restriction}{B}$ or $g{\restriction}{B}$ is commutative.	
\end{lemma}
\begin{proof}
The result is absolutely trivial if 	$g{\restriction}{B}$ is commutative, since in this case, by supportiveness of $f$, $h{\restriction}{B} = g{\restriction}{B}$. If on the other hand $f{\restriction}{B}$ is commutative then easily from the definition of $h$ follows that for any $x,y,z \in B$, $h(x,y,z) = h(y,z,x) = h(z,x,y).$ This form of superposition of $f$ and $g$  appears also in Bulatov \cite[Section 4.3]{bulatov2016conservative}.
\end{proof}
For notational convenience, we introduce the following definition:
\begin{definition}
Let $\bar{f}$ and $\bar{g}$ be two aggregators on $X$. Let $\bar{f} \diamond \bar{g}$ be    the ternary aggregator $\bar{h} = (h_1, \ldots, h_m)$ defined by:
\[h_j(x,y,z) = f_j(g_j(x,y,z), g_j(y,z,x), g_j(z,x,y)), j=1, \ldots,m, \]
 (The fact that $\bar{h} $ is indeed an aggregator follows from Lemma \ref{lem:superposition} and the fact that  a tuple of functions comprised of the same projections is an aggregator.)	
\end{definition}
\begin{lemma}\label{lem:diamond}
Let $\bar{f}$ and $\bar{g}$ be two aggregators on $X$. Let $i,j \in \{1, \ldots, m\}$ two arbitrary issues (perhaps identical) and $B_i, B_j$ two two-element subsets of $X_i$ and $X_j$, respectively. If $f_i{\restriction}{B_i}$ and $g_j{\restriction}{B_j}$ are commutative
 (i.e. by Lemma~\ref{lem:four_functions} if each is one of the  $\land^{(3)},\lor^{(3)}$, ${\rm maj}, \oplus$) 
 then both $\bar{f} \diamond \bar{g} {\restriction}{B_i}$ 
 and $\bar{f} \diamond \bar{g} {\restriction}{B_j}$	are commutative (i.e. each 
 is one of the $\land^{(3)},\lor^{(3)}$, ${\rm maj}, \oplus$).\end{lemma}
 \begin{proof}
 Immediate by Lemmas \ref{lem:four_functions} and \ref{lem:commutative}.	
 \end{proof}
 We now restate and prove the characterization of {\un} possibility domains:
 \begin{theorem}\label{thm:updcar}
Let $X$ be a set of feasible voting patterns. The following are equivalent:
\begin{enumerate}
	\item \label{item:one} $X$ is a {\un}  possibility domain.
	\item \label{item:two} For every $j=1, \ldots , m$ and for every two-element subset $B_j \subseteq X_j$, there is an aggregator $\bar{f} = (f_1, \ldots, f_m)$ (depending on $j$ and $B_j$) of some arity such that $f_j{\restriction{B_j}}$ is not a projection function.
	\item \label{item:three} There is ternary aggregator $\bar{f} = (f_1, \ldots, f_m)$ such that for all $j=1, \ldots, m$ and all two-element subsets $B_j \subseteq X_j$, $f_j{\restriction{B_j}} \in \{\land^{(3)}, \lor^{(3)}, {\rm maj}, \oplus\}$   (to which of the four  ternary operations $\land^{(3)}, \lor^{(3)}, {\rm maj}\text{ and }  \oplus$ the restriction $f_j{\restriction{B_j}}$ is equal to depends on $j$ and $B_j$).
	\item \label{item:four} There is ternary aggregator $\bar{f} = (f_1, \ldots, f_m)$ such that for all $j=1, \ldots, m$ and all $x,y \in X_j, f_j(x,y,y) = f_j(y,x,y) = f_j(y,y,x).$
	 \end{enumerate}
\end{theorem}

\begin{proof}
The directions \eqref{item:one}	 $\implies$ \eqref{item:two} and \eqref{item:three}	 $\implies$ \eqref{item:one} are obvious. Also the equivalence of \eqref{item:three} and \eqref{item:four} immediately follows from Lemma \ref{lem:four_functions}. It remains to show \eqref{item:two} $\implies$ \eqref{item:three}. For a two-element subset $B_j \subseteq X_j$,  let  ${\mathcal C}_{B_j}$ be the clone (Lemma \ref{lem:clone})  of the restrictions
$f_j{\restriction{B_j}}$ of the $j$-th components of aggregators $\bar{f} = (f_1, \ldots, f_m)$. By Post \cite{post1941two}, 
we can easily get that ${\mathcal C}_{B_j}$ contains one of the operations $\land, \lor, {\rm maj}\text{ and }  \oplus$. Therefore, easily,  for all $j, B_j$ there is a ternary aggregator $\bar{f} = (f_1, \ldots, f_m)$ 
(depending on $j,B_j$) such that 
$f_j{\restriction}{B_j}$ is one of the $\land^{(3)}, \lor^{(3)}, {\rm maj}\text{ and }  \oplus$. 
Now let $\bar{f}^1, \ldots, \bar{f}^N$ be an arbitrary enumeration of all   ternary aggregators each of which on some  issue $j$ and some two-element $B_j$ is one of the $\land^{(3)}, \lor^{(3)}, {\rm maj}\text{ and }  \oplus$ and  such that the $\bar{f}^l$'s cover all possibilities for $j,B_j$. 
As a ternary operation $\bar{h}$ such  that uniformly for each  $j,B_j$, the restriction  $h_j{\restriction}{B_j}$ belongs to the set  $\{\land^{(3)}, \lor^{(3)}, {\rm maj},  \oplus\}$ we can take, by Lemma~\ref{lem:diamond},  
\[(\cdots (\bar{f}^1 \diamond \bar{f}^2) \diamond \cdots \diamond \bar{f}^N),\]
which concludes the proof.
\end{proof}
 
Now Bulatov's dichotomy theorem \cite[Theorem 2.16]{bulatov2011complexity} in our setting reads:
\begin{dich-thm*}[Bulatov] If for any $j=1, \ldots, m$ and any two-element subset $B_j\subseteq X_j$ there is either a binary  aggregator $\bar{f} = (f_1,\ldots, f_m)$ 
such that $\f_j{\restriction}{B_j} \in \{ \land, \lor\}$ 
or a ternary aggregator $\bar{f} = (f_1,\ldots, f_m)$ such that $\f_j{\restriction}{B_j} \in \{ {\rm maj}, \oplus\}$, then $\text{MCSP}(\Gamma^{\rm cons}_X)$ is tractable; otherwise it is NP-complete.
\end{dich-thm*}

We finally restate and prove Theorem \ref{thm:upd}:

\begin{theorem}\label{thm:upd}
If $X$ is a uniformly possibility domain then $\text{MCSP}(\Gamma^{\rm cons}_X)$ is tractable; otherwise it is NP-complete.
\end{theorem}
\begin{proof}
The tractability part of the statement follows from Bulatov's 	Dichotomy Theorem and item \eqref{item:three} of Theorem \ref{thm:updcar} (observing that $x\land y = \land^{(3)}(x,x,y)$ and similarly for $\lor$ and using Lemma \ref{lem:superposition}), whereas the completeness part follows from Bulatov's 	Dichotomy Theorem and item \eqref{item:two} of Theorem \ref{thm:updcar}.
\end{proof}

\medskip
\noindent{\bf Acknowledgment}~ We are grateful to Mario Szegedy for sharing with us an early draft of his work 
on impossibility theorems and the algebraic toolkit.
We are also grateful to Andrei Bulatov for his bringing to our attention his ``three basic operations" proposition \cite[Proposition 3.1]{bulatov2011complexity}, \cite[Proposition 2.2]{bulatov2016conservative}.


\end{document}